\documentclass[12pt]{amsart}
\usepackage{amssymb,latexsym,verbatim}
\theoremstyle{plain}
\newtheorem{thm}{Theorem}[section]
\newtheorem{lem}[thm]{Lemma}
\newtheorem{cor}[thm]{Corollary}
\newtheorem{prop}[thm]{Proposition}

\theoremstyle{definition}

\theoremstyle{remark}

\newtheorem{rem}[thm]{Remark}

\DeclareMathOperator{\Ann}{Ann}
\DeclareMathOperator{\Tor}{Tor}
\DeclareMathOperator{\Ext}{Ext}
\DeclareMathOperator{\Supp}{Supp}
\DeclareMathOperator{\V}{V}
\DeclareMathOperator{\Hom}{Hom}
\DeclareMathOperator{\Ker}{Ker}
\DeclareMathOperator{\Coker}{Coker}
\DeclareMathOperator{\Image}{Im}

\DeclareMathOperator{\cd}{cd}

\DeclareMathOperator{\lc}{H}

\DeclareMathOperator{\G}{\Gamma}

\DeclareMathOperator{\ara}{ara}

\newcommand{\lo}{\longrightarrow}
\newcommand{\fa}{\mathfrak{a}}
\newcommand{\fb}{\mathfrak{b}}

\newcommand{\fm}{\mathfrak{m}}
\newcommand{\fp}{\mathfrak{p}}

\newcommand{\fc}{\mathfrak{c}}

\begin{document}

\title[Artinianness of local cohomology modules ]
{Artinianness of local cohomology modules  \\}

\author{Moharram Aghapournahr}
\address{ Moharram Aghapournahr\\  Arak University, Beheshti St, P.O. Box:879, Arak, Iran}
\email{m-aghapour@araku.ac.ir}

\author{Leif Melkersson}
\address{Leif Melkersson\\Department of Mathematics\\ Link\"{o}ping University\\ S-581 83 Link\"{o}ping\\ Sweden}

\email{lemel@mai.liu.se}



\begin{abstract}
 Let $A$ be a noetherian ring, $\fa$ an ideal of $A$, and $M$ an $A$--module. Some uniform theorems on the artinianness
 of certain local cohomology modules are proven in a general situation. They generalize and imply previous results about
  the artinianness of some special local cohomology modules in the graded case.
\end{abstract}

\maketitle

\section{Introduction}
Throughout $A$ is a commutative noetherian ring. As a general reference to homological and commutative algebra we use
\cite{BH} and \cite{Mat}. One of the main problems in the study of local cohomology modules is to determine when they
are artinian. Recently some results have been proved about the artinianness of graded local cohomology modules in
\cite{BFT}, \cite{BRS}, \cite{S1} and \cite{S2}.

We prove some uniform results about artinianness of local cohomology modules in the context of an arbitrary noetherian
ring $A$. Their proofs are simple and they have those special cases in the above references as immediate consequences.

A Serre subcategory of the category of $A$--modules is a full subcategory closed under taking submodules, quotient
 modules and extensions. An example is given by the class of artinian $A$--modules. A useful method to prove that
 a certain module belongs to such a Serre subcategory is to apply \cite[Lemma 3.1]{Mel}.

An $A$--module $M$ is called {\it $\fa$--cofinite} if $\Supp_A(M)\subset \V{(\fa)}$ and $\Ext^i_{A}(A/\fa,M)$ is
finite (finitely generated) for all $i$. This notion was introduced by Hartshorne in \cite{Ha}. For more information
 about cofiniteness with respect to an ideal, see \cite{HK}, \cite{DM} and \cite{Mel}.


\section{Main results}

\begin{thm}\label{P:loc/aloc}
Let $\mathcal S$ be a Serre subcategory of the category of $A$--modules. Let $M$ be a finite $A$ module and let $\fa$
 be an ideal of $A$ such that $\lc^{i}_{\fa}(M)$ belongs to $\mathcal S$ for all $i> n$. If $\fb$ is an ideal of $A$
 such that $\lc^{n}_{\fa}(M/{\fb}M)$ belongs to $\mathcal S$, then the module $\lc^{n}_{\fa}(M)/{\fb}\lc^{n}_{\fa}(M)$
 belongs to $\mathcal S$.
\end{thm}
\begin{proof}
Suppose $\lc^{n}_{\fa}(M)/{\fb}\lc^{n}_{\fa}(M)$ is not in $\mathcal S$. Let $N$ be a maximal submodule of $M$ such
that $\lc^{n}_{\fa}(M/N)\otimes_A{A/\fb}$ is not in $\mathcal S$. Let $L\supset N$ be such that $\G_{\fb}(M/N)=L/N$.
Since $\Supp_R(L/N)\subset{\V(\fb)\cap \Supp_R(M)}$,\ $\lc^{i}_{\fa}(L/N)$ belongs to $\mathcal S$ for all $i\geq n$
by \cite[Theorem 3.1]{AMel}.

From the exact sequence $0 \rightarrow L/N\rightarrow M/N\rightarrow M/L\rightarrow 0$, we get the exact sequence
\begin{center}
$\lc^{n}_{\fa}(L/N)\lo \lc^{n}_{\fa}(M/N)\overset{f}\lo \lc^{n}_{\fa}(M/L)\lo \lc^{n+1}_{\fa}(L/N)$.
\end{center}
$\Tor^A_{i}(A/\fb,\Ker f)$ and $\Tor^A_{i}(A/\fb,\Coker f)$ are in
$\mathcal S$ for all $i$, because $\Ker f$ and $\Coker f$ are in
$\mathcal S$. It follows from \cite[Lemma 3.1]{Mel}, that $\Ker
(f\otimes{A/\fb})$ and $\Coker (f\otimes{A/\fb})$ are in $\mathcal
S$. Since $\lc^{n}_{\fa}(M/N)\otimes_A{A/\fb}$ is not in $\mathcal
S$,
 the module $\lc^{n}_{\fa}(M/L)\otimes_A{A/\fb}$ can not be in $\mathcal S$. By the maximality of $N$, we get $L=N$.
  We have shown that $\G_{\fb}(M/N)=0$ and therefore we can take $x\in \fb$ such that the sequence
   $0\rightarrow M/N\overset{x}\rightarrow M/N\rightarrow M/(N+{x}M)\rightarrow 0$ is exact. Thus we get the exact
   sequence
\begin{center}
$\lc^{n}_{\fa}(M/N)\overset{x}\lo \lc^{n}_{\fa}(M/N)\lo \lc^{n}_{\fa}(M/N+{x}M)\lo \lc^{n+1}_{\fa}(M/N).$
\end{center}
 This yields the exact sequence
\begin{center}
$0\lo \lc^{n}_{\fa}(M/N)/{x}\lc^{n}_{\fa}(M/N)\lo \lc^{n}_{\fa}(M/N+{x}M)\lo C\lo 0$,
\end{center}
where $C\subset \lc^{n+1}_{\fa}(M/N)$ and thus $C$ is in $\mathcal S$.

Note that $x\in \fb$. Hence we get the exact sequence
\begin{center}
$\Tor^A_{1}(A/\fb,C)\lo \lc^{n}_{\fa}(M/N)\otimes_A{A/\fb}\lo \lc^{n}_{\fa}(M/(N+{x}M))\otimes_A{A/\fb}$
\end{center}
However $N\subsetneqq{(N+{x}M)}$ and therefore $\lc^{n}_{\fa}(M/(N+{x}M))\otimes_A{A/\fb}$ belongs to $\mathcal S$ by
the maximality of $N$. Consequently $\lc^{n}_{\fa}(M/N)\otimes_A{A/\fb}$ is in $\mathcal S$ which is a contradiction.
\end{proof}
\begin{cor}\label{T:H/bH}
Let $\fa$ and $\fb$ be two ideals of $A$. Let $M$ be a finite $A$--module. If $\lc^{i}_{\fa}(M)$ is artinian for
$i> n$ and $\lc^{n}_{\fa}(M/{\fb}M)$ is artinian, then $\lc^{n}_{\fa}(M)/{\fb}\lc^{n}_{\fa}(M)$ is artinian.
\end{cor}
\begin{proof}
In \ref{P:loc/aloc} take $\mathcal S$ as the category of artinian $A$--modules.
\end{proof}
\begin{cor}\label{C:H/bH}
Let $\fa$ and $\fb$ be two ideals of $A$ such that $A/{\fa}+\fb$ is artinian. If $\lc^{i}_{\fa}(M)$ is artinian for
 $i> n$, then $\lc^{n}_{\fa}(M)/{\fb}\lc^{n}_{\fa}(M)$ is artinian.
\end{cor}
\begin{proof}
Note that $\lc^{n}_{\fa}(M/{\fb}M)\cong \lc^{n}_{\fa+{\fb}}(M/{\fb}M)$ is artinian.
\end{proof}
\begin{cor}\label{C:H/aH}
If $\lc^{i}_{\fa}(M)$ is artinian for $i> n$, where $n\geq 1$ then \\$\lc^{n}_{\fa}(M)/{\fa}\lc^{n}_{\fa}(M)$ is
artinian.
\end{cor}
\begin{proof}
Note that $\lc^{n}_{\fa}(M/{\fa}M)= 0$ for $n\geq 1$
\end{proof}

\begin{rem}\label{R:H/aH}
In \ref{C:H/aH} we must assume that $n\geq 1$. Take any ideal $\fa$ in a ring $A$ such that $A/\fa$ is not artinian.
 Let $M=A/\fa$. Then $\lc^{i}_{\fa}(M)= 0$ for $i\geq 1$, and $\G_{\fa}(M)= M$. On the other hand $M/{\fa}M\cong M$.
  Thus $\lc^{0}_{\fa}(M)/{\fa}\lc^{0}_{\fa}(M)$ is not artinian.
\end{rem}

\begin{cor}\label{C:loc/aloc2}
Let $\fa$ and $\fb$ be two ideals of $A$. Let $M$ be a finite $A$--module. If $\lc^{i}_{\fa}(M)$ is
$\fa$--cofinite artinian {\rm (}resp. has finite support{\rm )} for $i> n$, and $\lc^{n}_{\fa}(M/{\fb}M)$ is
 $\fa$--cofinite artinian {\rm (}resp. has finite support{\rm )} then $\lc^{n}_{\fa}(M)/{\fb}\lc^{n}_{\fa}(M)$ is
  $\fa$--cofinite artinian {\rm (}resp. has finite support{\rm )}. In particular, if $n\geq 1$ then
   $\lc^{n}_{\fa}(M)/{\fa}\lc^{n}_{\fa}(M)$ has finite length {\rm (}resp. has finite support{\rm )}.
\end{cor}
\begin{cor}\label{C:aloc=loc}
If $c=\cd(\fa,M)> 0$, then $\lc^{c}_{\fa}(M)={\fa}\lc^{c}_{\fa}(M)$.
\end{cor}

As a corollary we recover Yoshida's theorem \cite[Proposition 3.1]{Yo}.
\begin{cor}\label{C:notfg}
$c=\cd(\fa,M)> 0$, then $\lc^{c}_{\fa}(M)$ is not finite.
\end{cor}
\begin{proof}
We may assume that $(A,\fm)$ is a local ring. Use corollary \ref{C:aloc=loc} and Nakayama's lemma.
\end{proof}
\begin{prop}\label{T:cofart1}
Let $\fa$ and $\fb$ be ideals of $A$ with $A/{\fa}+\fb$ artinian. If
the module $M$ is $\fa$--cofinite, then $\lc^{i}_{\fb}(M)$ is
artinian for all $i$.
\end{prop}
\begin{proof}
Use \cite[Theorem 5.5 (i)$\Leftrightarrow$(iv)]{Mel} and the fact that an $\fa$--cofinite module has finite Bass numbers.
\end{proof}

\begin{cor}
Let $\fa$ and $\fb$ be two ideals of $A$ such that $A/{\fa}+\fb$ is artinian. If $\fa$ is a principle ideal and $M$
is a finite module, then $\lc^{i}_{\fb}(\lc^{j}_{\fa}(M))$ is artinian for all $i, j$.
\end{cor}
\begin{proof}
Note that $\lc^{j}_{\fa}(M)$ is $\fa$--cofinite for all $j$ when $\fa$ is principle.
\end{proof}
\begin{thm}\label{T:dim<1}
Let $\fa$ and $\fb$ be two ideals of $A$ such that $A/{\fa}+\fb$ is artinian. Let $M$ be a finite module such that
$\dim M/{\fa}M\leq 1$. Then $\lc^{i}_{\fb}(\lc^{j}_{\fa}(M))$ is an artinian module for all $i, j$.
\begin{proof}
Since $A/{\fa}+\fb$ is artinian, $\Supp_A(\lc^{j}_{\fa}(M))\subset \V(\fa+{\fb})$ and $\V(\fa+{\fb})$ consists of
finitely many maximal ideals. Thus we may assume that $(A,\fm)$ is local.

Passing to the ring $A/\Ann(M)$ from \cite[Theorem 1]{DM} and the change of rings principle \cite [Corollary 2.6]{Mel},
 we get that $\lc^{j}_{\fa}(M)$ is $\fa$--cofinite for all $j$. We use theorem \ref{T:cofart1} to deduce that
  $\lc^{i}_{\fb}(\lc^{j}_{\fa}(M))$ is artinian for all $i,j$
\end{proof}
\end{thm}
\begin{thm}\label{T:dim<2}
Let $\fa$ and $\fb$ be two ideals of $A$ such that $A/{\fa}+\fb$ is artinian. Let $M$ be a finite module such that
$\dim M/{\fa}M\leq 2$. Then $\lc^{1}_{\fb}(\lc^{i}_{\fa}(M))$ is an artinian module for all $i$.
\end{thm}
\begin{proof}
We first reduce to the case $\G_{\fb}(M)=0$. From the exact sequence
\begin{center}
$0\lo \G_{\fb}(M)\lo M\overset{f}\lo \overline{M}\lo 0$,
\end{center}
where $\overline{M}= M/\G_{\fb}(M)$, we get the exact sequence
\begin{center}
$\lc^{i}_{\fa}(\G_{\fb}(M))\lo \lc^{i}_{\fa}(M)\overset{\lc^{i}_{\fa}(f)}\lo \lc^{i}_{\fa}(\overline{M})
\lo \lc^{i+1}_{\fa}(\G_{\fb}(M)).$
\end{center}
Note that $\lc^{i}_{\fa}(\G_{\fb}(M))\cong \lc^{i}_{\fa+{\fb}}(\G_{\fb}(M))$, and so is artinian for all $i$. Since
$\Ker \lc^{i}_{\fa}(f)$ and $\Coker \lc^{i}_{\fa}(f)$ are artinian, it follows from \cite[Lemma 3.1]{Mel} that
$\Ker \lc^{1}_{\fb}(\lc^{i}_{\fa}(f))$ and $\Coker \lc^{1}_{\fb}(\lc^{i}_{\fa}(f))$ are artinian. Hence
\begin{center}
$\lc^{1}_{\fb}(\lc^{i}_{\fa}(M))$ is artinian if and only if $\lc^{1}_{\fb}(\lc^{i}_{\fa}(\overline{M}))$ is artinian.
\end{center}
Thus we may assume that $\G_{\fb}(M)=0$.

Take $x\in \fb$ outside all associated prime ideals of $M$ and
outside all prime ideals $\fp\supset \fa+\Ann(M)$, such that
$\dim A/\fp= 2$. Then $\dim M/(\fa+{xA})M\leq 1$ and $x$ is a
non-zerodivisor on $M$. Hence we get the long exact sequence
\begin{center}
$\dots\rightarrow \lc^{i-1}_{\fa}(M/{x}M)\rightarrow \lc^{i}_{\fa}(M)\overset{x}\rightarrow \lc^{i}_{\fa}(M)\rightarrow
 \lc^{i}_{\fa}(M/{x}M)\rightarrow \dots$.
\end{center}
Consider the map $f$ which is multiplication by $x$ on $\lc^{i}_{\fa}(M)$. We have the exact sequence
\begin{center}
$0\lo \lc^{i-1}_{\fa}(M)/{x}\lc^{i-1}_{\fa}(M)\lo \lc^{i-1}_{\fa}(M/{x}M)\lo \Ker f\lo 0$
\end{center}
and hence the exact sequence
\begin{center}
$\lc^{1}_{\fb}(\lc^{i-1}_{\fa}(M/{x}M))\lo \lc^{1}_{\fb}(\Ker f)\lo
 \lc^{2}_{\fb}(\lc^{i-1}_{\fa}(M)/{x}\lc^{i-1}_{\fa}(M))$.
\end{center}
However the last term is zero since
$$\Supp_A(\lc^{i-1}_{\fa}(M)/{x}\lc^{i-1}_{\fa}(M))\subset \Supp_A(M/(\fa+{xA})M)$$
 and $\dim M/(\fa+{xA})M\leq 1$. Consequently $\lc^{1}_{\fb}(\Ker f)$ is artinian, since
  $\lc^{1}_{\fb}(\lc^{i-1}_{\fa}(M/{x}M))$ is artinian by theorem \ref{T:dim<1}. From the exactness of
  $0\rightarrow \Coker f\rightarrow \lc^{i}_{\fa}(M/{x}M)$, we get the exactness of
  $0\rightarrow \G_{\fb}(\Coker f)\rightarrow \G_{\fb}(\lc^{i}_{\fa}(M/{x}M))$. Again we use \ref{T:dim<1} to deduce
  that $\G_{\fb}(\lc^{i}_{\fa}(M/{x}M)$ is artinan . Hence $\G_{\fb}(\Coker f)$ is artinian.

 Since we now have shown that $\lc^{1}_{\fb}(\Ker f)$ and $\G_{\fb}(\Coker f)$ are artinian, we are able to apply
  \cite[Lemma 3.1]{Mel} to deduce that $\Ker \lc^{1}_{\fb}(f)$ is artinian. But $\lc^{1}_{\fb}(f)$ is just multiplication
   by $x$ on $\lc^{1}_{\fb}(\lc^{i}_{\fa}(M))$. It follows from \cite[Theorem 7.1.2]{BSh} that
   $\lc^{1}_{\fb}(\lc^{i}_{\fa}(M))$ is artinian.
\end{proof}
 The following lemma is a generalization of \cite[Corollary
 1.8]{Mel1}.
\begin{lem}\label{L:cofart}
Let $\fa$ and $\fb$ be two ideals of $A$ such that $A/{\fa}+\fb$ is artinian. If $M$ is a module such that
$\Supp_A(M)\subset \V(\fa)$ and $0:_M{\fa}$ is finite, then $\G_{\fb}(M)$ is $\fa$--cofinite artinian.
\end{lem}
\begin{proof}
The module $0:_{\G_{\fb}(M)}{\fa}\subset 0:_M{\fa}$ and is therefore finite. But the support of $0:_{\G_{\fb}(M)}{\fa}$
is contained in $\V({\fa+{\fb}})$, thus $0:_{\G_{\fb}(M)}{\fa}$ has finite length. Therefore $\G_{\fb}(M)$ is
$\fa$--cofinite artinian, by \cite[Proposition 4.1]{Mel}.
\end{proof}
\begin{cor}\label{C:cofart4}
Let $\fa$ and $\fb$ be two ideals of $A$ such that $A/{\fa}+\fb$ is artinian. If $M$ is a module such that
$$\Ext^n_{A}(A/\fa,M)\text{ and }\Ext^{n+1-j}_{A}(A/\fa, \lc^{j}_{\fa}(M))$$
 for all $j< n$, are finite, then $\G_{\fb}(\lc^{n}_{\fa}(M))$ is $\fa$--cofinite artinian.
\end{cor}
\begin{proof}
Note that by \cite[Theorem 6.3.9 (b)]{DY1} $\Hom_{A}(A/\fa,\lc^{n}_{\fa}(M))$ is finite, hence by lemma
\ref{L:cofart}, $\G_{\fb}(\lc^{n}_{\fa}(M))$ is $\fa$--cofinite artinian.
\end{proof}
\begin{cor}\label{C:cofart2}
Let $\fa$ and $\fb$ be two ideals of $A$ such that $A/{\fa}+\fb$ is artinian. If $M$ is a finite module such that
 $\lc^{i}_{\fa}(M)$ is $\fa$--cofinite for $i< n$, then $\G_{\fb}(\lc^{i}_{\fa}(M))$ is $\fa$--cofinite artinian for
  all $i\leq n$.
\end{cor}
\begin{thm}\label{T:cofart3}
Let $\fa$ and $\fb$ be two ideals of $A$ such that $A/{\fa}+\fb$ is artinian. For each finite module $M$, the modules
 $\G_{\fb}(\lc^{1}_{\fa}(M))$ and $\lc^{1}_{\fb}(\lc^{1}_{\fa}(M))$ are artinian.
\end{thm}
\begin{proof}
Corollary \ref{C:cofart2} with $n=1$ implies that $\G_{\fb}(\lc^{1}_{\fa}(M))$ is artinian. We may assume that
 $\G_{\fa}(M)=0$, so there is an $M$--regular element $x$ in $\fa$. From the exact sequence $0\rightarrow
 M\overset{x}\rightarrow M\rightarrow M/{x}M \rightarrow 0$, we get the exact sequence
\begin{equation}\label{E:exact}
0\lo \G_{\fa}(M/{x}M)\lo \lc^{1}_{\fa}(M)\overset{x}\lo \lc^{1}_{\fa}(M)\lo \lc^{1}_{\fa}(M/{x}M).
\end{equation}

Consider the map $f$ defined as multiplication with $x$ on $\lc^{1}_{\fa}(M)$ occuring in (\ref{E:exact}). We get
\begin{center}
$\lc^{1}_{\fb}(\Ker f)\cong \lc^{1}_{\fb}(\G_{\fa}(M/{x}M))\cong \lc^{1}_{\fa+{\fb}}(\G_{\fa}(M/{x}M))$,
\end{center}
which is artinian and the exact sequence
\begin{center}
$0\lo \G_{\fb}(\Coker f)\lo \G_{\fb}(\lc^{1}_{\fa}(M/{x}M))$
\end{center}
Since $\G_{\fb}(\lc^{1}_{\fa}(M/{x}M))$ artinian by \ref{C:cofart2}, $\G_{\fb}(\Coker f)$ is artinian. We can use
\cite[Lemma 3.1]{Mel} with $S=\G_{\fb}(-)$ and $T=\lc^{1}_{\fb}(-)$ to conclude that $\Ker \lc^{1}_{\fb}(f)$ is artinian.
 But $\lc^{1}_{\fb}(f)$ is multiplication by the element $x\in \fa$ on $\lc^{1}_{\fb}(\lc^{1}_{\fa}(M))$. Again using
  \cite[Theorem 7.1.2]{BSh}, we conclude that $\lc^{1}_{\fb}(\lc^{1}_{\fa}(M))$ is artinian.
\end{proof}
\begin{lem}\label{L:artkercok}
Let $f:L\lo M$ be $A$--linear and $\fc$ an ideal of $A$. Let
$\mathcal S$ be a Serre subcategory of the category of
$A$--modules. If $\Ker \lc^{j}_{\fc}(f)$ and $\Coker
\lc^{j}_{\fc}(f)$ are in $\mathcal S$ for all $j$, then for each
$i$:
\begin{center}
$\lc^{i+2}_{\fc}(\Ker f)$ is in $\mathcal S$ if and only if $\lc^{i}_{\fc}(\Coker f)$ is in $\mathcal S$.
\end{center}
\end{lem}
\begin{proof}
Let $K=\Ker f,\, I=\Image f$ and $C=\Coker f$. Factorize $f$ as $f=h\circ{g}$, where
\begin{center}
$0\lo K\lo L\overset{g}\lo I\lo 0$
\end{center}
\begin{center}
and
\end{center}
\begin{center}
$0\lo I\overset{h}\lo M\lo C\lo 0$
\end{center}
are exact. For each $j$, we have the exact sequence
\begin{center}
$0\rightarrow \Ker \lc^{j}_{\fc}(g)\rightarrow \Ker \lc^{j}_{\fc}(f)\rightarrow \Ker \lc^{j}_{\fc}(h)\rightarrow
 \Coker \lc^{j}_{\fc}(g)\rightarrow \Coker \lc^{j}_{\fc}(f)\rightarrow \Coker \lc^{j}_{\fc}(h)\rightarrow 0.$
\end{center}
The hypothesis implies that $\Ker \lc^{j}_{\fc}(g)$ and $\Coker
\lc^{j}_{\fc}(h)$ are in $\mathcal S$ for each $j$. For each $j$,
$\Ker \lc^{j}_{\fc}(h)$ is in $\mathcal S$ if and only if $\Coker
\lc^{j}_{\fc}(g)$ is in $\mathcal S$. Moreover there are exact
sequences
\begin{center}
$\lc^{i}_{\fc}(I)\overset{\lc^{i}_{\fc}(h)}\lo \lc^{i}_{\fc}(M)\lo \lc^{i}_{\fc}(C)\lo
\lc^{i+1}_{\fc}(I)\overset{\lc^{i+1}_{\fc}(h)}\lo \lc^{i+1}_{\fc}(M)$
\end{center}
\begin{center}
and
\end{center}
\begin{center}
$\lc^{i+1}_{\fc}(L)\overset{\lc^{i+1}_{\fc}(g)}\lo \lc^{i+1}_{\fc}(I)\lo \lc^{i+2}_{\fc}(K)\lo
 \lc^{i+2}_{\fc}(L)\overset{\lc^{i+2}_{\fc}(g)}\lo \lc^{i+2}_{\fc}(I).$
\end{center}
It follows that
$$\begin{matrix}
\lc^{i}_{\fc}(C)\in \mathcal S&\Leftrightarrow &\Ker \lc^{i+1}_{\fc}(h)\in \mathcal S\\
       &\Leftrightarrow &\Coker \lc^{i+1}_{\fc}(g)\in \mathcal S\\
       &\Leftrightarrow &\lc^{i+2}_{\fc}(K)\in \mathcal S&
\end{matrix}$$.
\end{proof}
\begin{thm}\label{T:cofart5}
Let $\fa$ and $\fb$ be two ideals of $A$ such that $A/{\fa}+\fb$ is artinian. Assume that $\ara(\fa)=2$. Then for each
 finite $A$--module $M$, the module $\lc^{i}_{\fb}(\lc^{2}_{\fa}(M))$ is artinian if and only if the module
 $\lc^{i+2}_{\fb}(\lc^{1}_{\fa}(M))$ is artinian.
\end{thm}
\begin{proof}
We may assume that $\G_{\fa}(M)=0$. Then $\fa$ can be generated
by $M$--regular elements $x$ and $y$ on $M$. (See \cite[Exercise
16.8]{Mat}). Then there is an exact sequence \cite[Proposition
8.1.2]{BSh}
\begin{center}
$0\lo \lc^{1}_{\fa}(M)\lo \lc^{1}_{xA}(M)\overset{f}\lo \lc^{1}_{xA}(M_y)\lo \lc^{2}_{\fa}(M)\lo 0.$
\end{center}
In order to apply lemma \ref{L:artkercok} we prove that $\Ker \lc^{j}_{\fb}(f)$ and $\Coker \lc^{j}_{\fb}(f)$ are
artinian for all $j$. By \cite[Proposition 8.1.2]{BSh}, there is an exact sequence
\begin{center}
$\lc^{j}_{\fb+{y}A}(\lc^{1}_{xA}(M))\lo \lc^{j}_{\fb}(\lc^{1}_{xA}(M))\overset{\lc^{j}_{\fb}(f)}\lo
 \lc^{j}_{\fb}(\lc^{1}_{xA}(M)_y)\lo \lc^{j+1}_{\fb+{y}A}(\lc^{1}_{xA}(M)).$
\end{center}
Since $\lc^{1}_{xA}(M)$ is ${xA}$--cofinite, the outer modules are artinian, by \ref{T:cofart1}. Hence
$\Ker \lc^{j}_{\fb}(f)$ and $\Coker \lc^{j}_{\fb}(f)$ are artinian for all $j$.
\end{proof}

\providecommand{\bysame}{\leavevmode\hbox
to3em{\hrulefill}\thinspace}

\end{document}